\documentclass[11pt]{amsart}

\usepackage{latexsym}
\usepackage{amssymb}
\usepackage{amsmath}
\usepackage{xcolor}
\usepackage{tikz-cd}
\usepackage{mathrsfs}

\usepackage[italian, english]{babel}
\usepackage{url}

\newtheorem{theorem}{Theorem}[section]
\newtheorem{lemma}[theorem]{Lemma}

\theoremstyle{definition}

\theoremstyle{remark}

\newtheorem{remark}[theorem]{Remark}

\def\N{\mathbb{N}}

\def\G{\mathcal{G}}

\def\U{\mathcal{U}}
\def\V{\mathcal{V}}
\def\W{\mathcal{W}}
\def\Zc{\mathcal{Z}}

\begin{document}

\title[A new ultrafilter proof of Van der Waerden's Theorem]
{A new ultrafilter proof 
\\
of Van der Waerden's Theorem}

\author{Mauro Di Nasso}

\address{Dipartimento di Matematica\\
Universit\`a di Pisa, Italy}

\email{mauro.di.nasso@unipi.it}




\begin{abstract}
We present a new short proof of Van der Waerden's Theorem
about the existence of arbitrarily long monochromatic arithmetic progressions.
The proof uses algebra in the compact space of ultrafilters $\beta\N$,
but contrarily to the other existing proofs, neither minimal nor idempotent ultrafilters are involved.
\end{abstract}

\maketitle

\section*{Introduction}

Van der Waerdens' Theorem is one of the fundamental results of combinatorics,
with a large number of applications across mathematics. 
It states that for every finite partition (coloring) $\N=C_1\cup\ldots\cup C_r$
there exist arbitrarily long ``monochromatic" arithmetic progressions
$a, a+d,\ldots,a+(\ell-1)d\in C_i$.
After the original purely combinatorial proof 
based on an ingenious and rather complex double inductive procedure \cite{vdw}, 
several others have been found over the years (for example,
a remarkable proof was found by S. Shelah in \cite{sh}).

An ultrafilter proof was first found in 1989 by 
V. Bergelson, H. Furstenberg, N. Hindman, and Y. Katznelson \cite{bfhk}
using minimal idempotent ultrafilters in the right topological compact semigroup $(\beta\N,\oplus)$; 
in fact, every set that belongs to a minimal ultrafilter
contains arbitrarily long arithmetic progressions, 
as was proven shortly thereafter by V. Bergelson and N. Hindman \cite{bh}.
About fifteen years later, in 2004, S. Koppelberg \cite{ko} found another proof 
of Van der Waerden's theorem via ``retractions," again using minimal idempotent ultrafilters. 
In 2020, the author \cite{dn} found a proof based on a special class of filters, 
called ``translation invariants," which are closely related to minimal ultrafilters. 

In this paper, we present a new and rather brief ultrafilter proof.
We proceed with a simple induction, considering at each step 
suitable ultrafilters on the Cartesian square $\N\times\N$ that 
``witness" the inductive hypothesis, and we use algebra in $(\beta\N,\oplus)$. 
It is worth noting that neither minimal ultrafilters nor idempotent ultrafilters 
are used in the construction.

\section{Preliminaries}

A family of ``patterns" $\G\subseteq\mathcal{P}(\N)$ is called
\emph{partition regular} if for every finite partition (coloring) $\N=C_1\cup\ldots\cup C_r$
there exists a color $C_i$ and a pattern $G\in\G$ such that $G\subseteq C_i$.
In this case, we say that $G$ is \emph{monochromatic} for the partition considered.
It is well known that partition regularity is closely related to ultrafilters.
In fact, partition regularity holds if and only if there exists an ultrafilter that is 
``witness" to it, in the sense that each of its elements includes one of the patterns of the family.

\begin{theorem}
Let $\G\subseteq\mathcal{P}$ be a family of patterns.
The following two properties are equivalent:
\begin{enumerate}
\item
For every finite coloring $\N=C_1\cup\ldots\cup C_r$ there exists a 
pattern $G\in\G$ which is monochromatic, \emph{i.e.} $G\subseteq C_i$ for some $i$.
\item
There exists an ultrafilter $\U$ such that for every $A\in\U$
there exists $G\in\G$ with $G\subseteq A$.
\end{enumerate}
\end{theorem}

\begin{proof}
See \cite{hs} Theorem 5.7.
\end{proof}

For our proof, we will need a variant of the above characterization that focuses 
on arithmetic progressions and considers ultrafilters on the Cartesian square $\N\times\N$.

\smallskip
Before proceeding, we list below some well-known fundamental notions and facts
regarding ultrafilters that we will use in the following
(a natural reference is the comprehensive monograph \cite{hs}).

\smallskip
\begin{itemize}
\item
If $\U$ is an ultrafilter on a set $I$ and $f:I\to J$,
then the \emph{image ultrafilter} $f(\U)$ on $J$ is defined by
setting $A\in f(\U)\Leftrightarrow f^{-1}(A)\in\U$ for every $A\subseteq J$.
\item
If $\U$ and $\V$ are ultrafilters on the sets $I$ and $J$ respectively,
the \emph{tensor product} $\U\otimes\V$ is the ultrafilter on the
Cartesian product $I\times J$ defined by setting
$X\in\U\otimes\V\Leftrightarrow\{i\in I\mid\{j\in J\mid (i,j)\in X\}\in\V\}\in\U$
for every $X\subseteq I\times J$. 
\item
The \emph{pseudo-sum} $\U\oplus\V$ between ultrafilters $\U, \V$ on $\N$ is defined by
setting $A\in\U\oplus\V\Leftrightarrow\{n\mid A-n\in\V\}\in\U$ for every $A\subseteq\N$,
$A-n=\{m\in\N\mid m+n\in A\}$ is the leftward shift of $A$ by $n$.
\item
Tensor products are associative,
provided one identifies Cartesian products $(I\times J)\times K$ 
with $I\times(J\times K)$, but not commutative.
\item
The pseudo-sum operation is associative but not commutative.
\item
Pseudo-sums $\U_1\oplus\ldots\oplus\U_k=\text{Sum}(\U_1\otimes\ldots\otimes\U_k)$ 
are the image ultrafilters of 
tensor products under the sum function $\text{Sum}(n_1,\ldots,n_k)\mapsto n_1+\ldots+n_k$.
\end{itemize}
 
The following property derives directly from the definitions; 
verification is straightfoward.

\begin{lemma}\label{tensorproducts}
For every $s=1,\ldots,k$, let $\U_s$ be an ultrafilter on $I_s$
and let $f_s:I_s\to J_s$ be a function. Then 
$$f_1(\U_1)\otimes\ldots\otimes f_k(\U_k)=(f_1,\ldots,f_k)(\U_1\otimes\ldots\otimes\U_k)$$
where $(f_1,\ldots,f_k):(n_1,\ldots,n_k)\mapsto(f_1(n_1),\ldots,f_k(n_k))$.
\end{lemma}

Finally, here is the characterization we will need.

\begin{lemma}\label{VdW-ultra}
For every $\ell\in\N$ the following two properties are equivalent:
\begin{enumerate}
\item
For every finite coloring $\N=C_1\cup\ldots\cup C_r$ there exists a 
a monochromatic $\ell$-term arithmetic progression 
$a, a+d, \ldots, a+(\ell-1)d\in C_i$.
\item
There exists an ultrafilter $\W$ on $\N\times\N$ such that
$T_0(\W)=T_j(\W)$ for all $j=0,1,\ldots,\ell-1$, where $T_j:(n,m)\mapsto n+jm$.
\end{enumerate}
\end{lemma}

\begin{proof}
$(2)\Rightarrow(1)$. Given a finite coloring $\N=C_1\cup\ldots\cup C_r$, pick the 
color $C=C_i$ that belongs to the ultrafilter $\U:=T_0(\W)=T_1(\W)=\ldots=T_{\ell-1}(\W)$.
Since $T_j^{-1}(C)\in\W$ for every $j$, we can take an
element $(a,d)\in\bigcap_{j=0}^{\ell-1}T_j^{-1}(C)\in\W$.
The $\ell$-term arithmetic progression $a, a+d, \ldots, a+(\ell-1)d$
is included in $C$.

\smallskip
$(1)\Rightarrow(2)$. For $A\subseteq\N$, let 
\begin{multline*}
X(A):=\{(a,d)\in\N\times\N\mid a, a+d, \ldots, a+(\ell-1)d\in A\ \text{or}\ 
\\
a, a+d, \ldots, a+(\ell-1)d\in A^c\}.
\end{multline*}
We claim that the family $\mathcal{G}=\{X_A\mid A\subseteq\N\}$ 
has the finite intersection property.
To see this, given $A_1,\ldots,A_n\subseteq\N$,
for every function $\chi:\{1,\ldots,n\}\to\{+,-\}$
we define $C_\chi:=\bigcap_{i=1}^n (A_i)^{\chi(i)}$,
where we agree that $A^+=A$ and $A^-=A^c$ is the complement. 
Then we consider the finite coloring $\N=\bigcup_{\chi}C_\chi$.
By the hypothesis, there exists a monochromatic $\ell$-term arithmetic progression
$a, a+d,\ldots,a+(\ell-1)d\in C_\chi$ for a suitable $\chi$.
By the definitions, it is readily seen that the pair $(a,d)\in X(A_i)$ for every $i=1,\ldots,n$.

Now let us take an ultrafilter $\W$ on $\N\times\N$ that extends $\mathcal{G}$.
If, by contradiction, $T_0(\W)\ne T_j(\W)$ for some $1\le j\le \ell-1$, then 
we could choose a set $A\in T_0(\W)$ such that its complement $A^c\in T_j(\W)$.
In this case, the set $Y:=T_0^{-1}(A)\cap T_j^{-1}(A^c)$ would belong to $\W$.
Finally, we observe that $Y\cap X_A=\emptyset$, and we obtain 
the desired contradiction because $X_A\in\W$.
\end{proof}

\section{The ultrafilter proof}

Before presenting the proof, let us introduce a convenient notation for 
iterated tensor products and pseudo-sums. For $n\in\N$ and $\U\in\beta\N$, let:
$$\U^{n\otimes}=\underbrace{\U\otimes\ldots\otimes\U}_{n\ \text{times}}\ ;\quad
\U^{n\oplus}=\underbrace{\U\oplus\ldots\oplus\U}_{n\ \text{times}}.$$
We agree that $\U^{1\otimes}=\U^{1\oplus}=\U$.

\begin{theorem}[Van der Waerden]
For every $\ell\in\N$ and for every finite coloring $\N=C_1\cup\ldots\cup C_r$
there exists a monochromatic $\ell$-term arithmetic progression $a, a+d,\ldots,a+(\ell-1)d\in C_i$.
\end{theorem}

\begin{proof}
We proceed by induction on the length $\ell$ of the arithmetic progression.
The base cases $\ell=1,2$ are trivial.
At the inductive step, consider an arbitrary finite
coloring $\N=C_1\cup\ldots\cup C_r$.
By the inductive hypothesis, we can pick an ultrafilter $\W$ as given by Lemma \ref{VdW-ultra}.
Let $\U:=T_\ell(\W)$, $\V:=T_0(\W)=T_1(\W)=\ldots=T_{\ell-1}(\W)$, and 
consider the following ultrafilters on $\N$:
\begin{itemize}
\item
$\Zc_s:=\U^{s\oplus}\oplus \V^{(r+2-s)\oplus}$ for $s=1,\ldots,r+1$.
\end{itemize}
By the property of ultrafilters, each of the above $r+1$ ultrafilters
contains a color of the partition and hence, by the pigeonhole principle, 
there must be a color $C=C_i$ that belongs to two of them,
\emph{i.e.}, $C\in\Zc_p\cap\Zc_q$ where $p<q$.
Now let:
\begin{itemize}
\item
$\Gamma:=\{k\in\N\mid C-k\in\V^{(r+2-q)\oplus}\}$.
\end{itemize}
Since $C\in\Zc_p=\U^{p\oplus}\oplus\V^{(q-p)\oplus}\oplus\V^{(r+2-q)\oplus}$, 
we have $\Gamma\in\U^{p\oplus}\oplus\V^{(q-p)\oplus}$; 
and since $C\in\Zc_q=\U^{p\oplus}\oplus\U^{(q-p)\oplus}\oplus\V^{(r+2-q)\oplus}$, we have
$\Gamma\in\U^{p\oplus}\oplus\U^{(q-p)\oplus}$. Then:
\begin{itemize}
\item
$\Gamma_1:=\{k'\in\N\mid \Gamma-k'\in\V^{(q-p)\oplus}\}\in\U^{p\oplus}$, and
\item
$\Gamma_2:=\{k'\in\N\mid \Gamma-k'\in\U^{(q-p)\oplus}\}\in\U^{p\oplus}$.
\end{itemize}
Take any $k'\in\Gamma_1\cap\Gamma_2\in\U^{p\oplus}$. Then
$\Gamma-k'\in \V^{(q-p)\oplus}\cap\U^{(q-p)\oplus}$, and hence
$\Gamma-k'\in\bigcap_{j=0}^\ell T_j(\W)^{(q-p)\oplus}$.
By using Lemma \ref{tensorproducts}, we observe that for every $j=0,\ldots,\ell$ one has
\begin{multline*}
T_j(\W)^{(q-p)\oplus}=\text{Sum}(T_j(\W)^{(q-p)\otimes})=
\\
=\text{Sum}\left((T_j,\ldots,T_j)(\W^{(q-p)\otimes})\right)=\psi_j(\W^{(q-p)\otimes})
\end{multline*}
where $\psi_j:=\text{Sum}\circ(T_j,\ldots,T_j):(\N\times\N)^{q-p}\to\N$ is the function such that
$$\psi_j:((n_1,m_1),\ldots,(n_{q-p},m_{q-p}))\longmapsto\sum_{s=1}^{q-p}(n_s+jm_s).$$
Since $\Gamma-k'\in\bigcap_{j=0}^\ell\psi_j(\W^{(q-p)\otimes})$, we have that
$\bigcap_{j=0}^\ell\psi_j^{-1}(\Gamma-k')\in\W^{(q-p)\otimes}$.
If $((n_1,m_1),\ldots(n_{q-p},m_{q-p}))$ is any element in that intersection,
then we have $k'+\sum_{s=1}^{q-p}(n_s+jm_s)\in\Gamma$ for every $j=0,\ldots\ell$,
and so we can pick 
$$k\in\bigcap_{j=0}^\ell C-(k'+\sum_{s=1}^{q-p}(n_s+jm_s))\in\V^{(r+2-q)\oplus}.$$
It follows that
$k+k'+\sum_{s=1}^{q-p}(n_s+jm_s)\in C$ for $j=0,\ldots,\ell$.
By letting $a:=k+k'+\sum_{s=1}^{q-p}n_s$ and $d:=\sum_{s=1}^{q-p}m_s$,
we see that $a, a+d, \ldots, a+\ell d\in C$ is
the desired monochromatic $(\ell+1)$-term arithmetic progression. 
\end{proof}

\begin{remark}
The above proof was first obtained by using iterated nonstandard extensions 
(see \cite{dj} for the foundations and some applications of that nonstandard technique). 
As its reformulation in terms of ultrafilters turned out to be 
rather simple and seems to have some original aspects, 
we decided to present it here in that language.
\end{remark}

\bibliographystyle{amsalpha}

\end{document}